\documentclass[12pt, a4paper]{article}
\usepackage{amsmath, amssymb,amsthm}
\newcommand{\lcm}{\mathrm{lcm}}
\newtheorem{theorem}{Theorem}[section]
\newtheorem*{maintheorem}{Theorem 1.1}
\newtheorem{lemma}[theorem]{Lemma}
\newtheorem{prop}[theorem]{Proposition}
\newtheorem{cor}[theorem]{Corollary}

\theoremstyle{remark}

\newtheorem{remark}[theorem]{Remark}
\makeatletter
\begin{document}
\baselineskip=20pt
\title{A uniqueness of periodic maps on surfaces}
\author{Susumu Hirose\thanks{This research was partially supported 
by Grant-in-Aid for Scientific Research (C) (No.24540096), 
Japan Society for the Promotion of Science. }, 
Yasushi Kasahara\thanks{This research was partially supported 
by Grant-in-Aid for Scientific Research (C) (No.23540102), 
Japan Society for the Promotion of Science. }}
\date{}
\maketitle
\begin{abstract}
Kulkarni showed that, if $g$ is greater than $3$, 
a periodic map on an oriented surface $\Sigma_g$ of genus $g$ 
with order more than or equal to $4g$ is uniquely determined by its order, 
up to conjugation and power.  
In this paper, we show that, if $g$ is greater than $30$, 
the same phenomenon happens for periodic maps 
on the surfaces with orders more than $8g/3$ 
and, for any integer $N$, 
there is $g > N$ such that there are periodic maps of $\Sigma_g$ of 
order $8g/3$ which are not conjugate up to power each other.  
Moreover, as a byproduct of our argument, 
we provide a short proof of Wiman's classical theorem: 
the maximal order of periodic maps of $\Sigma_g$ is $4g+2$. 
\end{abstract}

\section{Introduction}
Let $\Sigma_g$ be the oriented closed surface of 
genus $g \geq 2$. 
By the Nielsen-Thurston theory \cite{Thurston}, 
orientation preserving homeomorphisms of $\Sigma_g$ are classified into 3-types: 
(1) periodic, (2) reducible, (3) pseudo-Anosov. 
For each type, there are important values describing 
conjugacy classes, for example, the {\em orders\/} of periodic maps. 
Kulkarni \cite{Kulkarni} showed that 
if the genus $g$ is sufficiently large and the order is more than 
or equal to $4g$ 
then the order determines the conjugacy class of the periodic map 
up to power. 
The first author \cite{Hirose} showed the same type of result 
when the order is more than or equal to $3g$. 
In this paper, we investigate on the minimum $M$ 
satisfying the following condition: 
{\em if the genus $g$ is sufficiently large and $n > Mg$ (or $n \geq Mg$) 
then the order determines the conjugacy class of the periodic map 
up to power.\/}

\begin{theorem}\label{thm:main}
Let $g > 30$, and $n > 8g/3$.  
If there is a periodic map of $\Sigma_g$ of order $n$, then 
this map is unique up to conjugacy and power. 
On the other hand, let $N$ be any positive integer. 
There is $g > N$ such that there are periodic maps of $\Sigma_g$ of 
order $8g/3$ 
which are not conjugate up to power each other. 
\end{theorem}

We explain the outline of the proof of Theorem \ref{thm:main}. 
By \cite{Kasahara}, the periodic map which satisfies the condition 
of Theorem \ref{thm:main} is irreducible, that is, the orbit surface of this 
periodic map is a 2-sphere with 3 branched points. 
Let $n_1$  be the minimum of the branching indices. 
In \S 3, it is shown that the order is included in one of certain disjoint
ranges which is determined solely by the value of $n_1$ when the genus is sufficiently large. 
By using this result, we observe that $n_1$ should be 
at most $4$ under the condition in Theorem \ref{thm:main}. 
In \S 4, we discuss the uniqueness of periodic map by the order 
up to conjugacy and power. 
\par

It seems that our argument in \S 3, especially Theorem \ref{thm:n1-order-1}, is also
useful for several known results on the distribution of periodic maps, in simplifying
the subcases to be considered in their proofs.
As an example, in \S 5, we provide a short and complete proof of Wiman's classical theorem: 
the maximal order of periodic maps of $\Sigma_g$ is $4g+2$. 
\par

After we finished to write this paper, we were 
informed from Professor G. Gromadzki 
about their preprint \cite{BCGH} and that Theorem \ref{thm:main} 
follows directly from their result in \cite{BCGH}.

\section{Preliminaries}
An orientation preserving homeomorphism $f$ from a surface $\Sigma_g$ to itself 
is said to be a {\em periodic map\/}, 
if there is a positive integer $n$ such that $f^n = \mathrm{id}_{\Sigma_g}$. 
The {\em order\/} of $f$ is the smallest positive integer which satisfies the above 
condition. 
Two periodic maps $f$ and $f'$ on $\Sigma_g$ are {\em conjugate\/}, 
if there is an orientation preserving homeomorphism $h$ from $\Sigma_g$ to itself 
such that $f' = h \circ f \circ h^{-1}$. 
In this section, we will review the classification of conjugacy classes of 
periodic maps on surfaces by Nielsen \cite{Nielsen}. 
We follow a description by Smith \cite{Smith} and Yokoyama \cite{Yokoyama}. 

Let $f$ be a periodic map on $\Sigma_g$, whose order is $n$. 
A point $p$ on $\Sigma_g$ is a {\em multiple point\/} of $f$, if 
there is a positive integer $k$ less than $n$ such that $f^k(p) = p$. 
Let $M_f$ be the set of multiple points of $f$. 
The orbit space $\Sigma_g/f$ of $f$ is defined by identifying $x$ in $\Sigma_g$ 
with $f(x)$. Let $\pi_f : \Sigma_g \to \Sigma_g/f$ be the quotient map. 
Then $\pi_f$ is an $n$-fold branched covering ramified at $\pi_f (M_f)$. 
The set $\pi_f(M_f)$ is denoted by $B_f$, and each element of $B_f$ is 
called a {\em branch point\/} of $f$. 
We choose a point $x$ in $\Sigma_g/f - B_f$, and a point $\tilde{x}$ in 
$\pi_f^{-1}(x)$. 
We define a homomorphism $\Omega_f : \pi_1 (\Sigma_g /f - B_f) \to \mathbb{Z}_n$ 
as follows: 
Let $l$ be a loop in $\Sigma_g / f -B_f$ with the base point $x$, and 
$[l]$ the element of $\pi_1 (\Sigma_g/f - B_f)$ represented by $l$. 
Let $\tilde{l}$ be the lift of $l$ on $\Sigma_g$ which begins from $\tilde{x}$. 
There is a positive integer $r$ less than or equal to $n$ such that 
the terminal point of $\tilde{l}$ is $f^r(\tilde{x})$. 
We define $\Omega_f([l]) = r \mod n$. 
Since $\mathbb{Z}_n$ is an Abelian group, the homomorphism $\Omega_f$ 
induces a homomorphism $\omega_f$ from the abelianization of 
$\pi_1 (\Sigma_g/f - B_f)$ to $\mathbb{Z}_n$. 
The abelianization of $\pi_1 (\Sigma_g/f - B_f)$ is $H_1(\Sigma_g/f - B_f)$, 
therefore $\omega_f$ is a homomorphism from $H_1(\Sigma_g/f - B_f)$ to $\mathbb{Z}_n$. 
For each point of $B_f = \{ Q_1, \ldots, Q_b \}$, let $D_i$ be a disk in 
$\Sigma_g / f$, which contains $Q_i$ in its interior and is sufficiently 
small so that no other points of $B_f$ are in $D_i$. 
Let $S_{Q_i}$ be the boundary of $D_i$ with clockwise orientation. 
\begin{theorem}\cite[\S 11]{Nielsen} \label{thm:Nielsen}
%
Two periodic maps $f$ and $f'$ on $\Sigma_g$ are conjugate to each other 
if and only if the following three conditions are satisfied. 

\noindent
$(1)$ The order of $f$ is equal to the order of $f'$.  

\noindent
$(2)$ The number of elements in $B_f$ is equal to that of $B_{f'}$. 

\noindent
$(3)$ After renumbering the elements of $B_{f'}$, we have $\omega_f (S_{Q_i}) = \omega_{f'} (S_{Q_i})$ 
for each $i$. 
\end{theorem}
Let $\theta_i = \omega_f(S_{Q_i})$ for each $i$. 
By the above Theorem, the data 
$[g, n ; \theta_1, \ldots, \theta_b ]$ determines a periodic map 
up to conjugacy. 
The following proposition shows a sufficient and necessary condition for 
a data $[g, n ; \theta_1, \ldots, \theta_b ]$ to correspond to a periodic map. 
\begin{prop} \label{prop:cond-data}
%
There is a periodic map with the data $[g, n ; \theta_1, \ldots, \theta_b]$ 
if and only if the following conditions are satisfied. 

\noindent
$(1)$ $\theta_1 + \cdots + \theta_b \equiv 0 \mod n$. 

\noindent
$(2)$ Let  $n_i = n/\gcd \{ \theta_i, n \}$, then there exists a non-negative integer $g'$ which 
satisfies 
$$2g-2=n \left( 2 g' -2 + \sum
\left( 1-\frac{1}{n_i} \right) \right), $$
where $i$ runs through the branch points. 

\noindent
$(3)$ If $g' = 0$, then 
$\gcd \{\theta_1, \ldots, \theta_b \} \equiv 1 \mod n$. 

\end{prop}
The necessity of three conditions in the above Proposition are shown 
as follows.  
(1) follows from the fact that $\omega_f$ is a homomorphism and 
$S_{Q_1} + \cdots + S_{Q_b}$ is null-homologous, 
(2) is the Riemann-Hurwitz formula, and (3) follows from the fact 
that $\omega_f$ is a surjection.
The sufficiency of these conditions follows from the existence theorem of 
a branched covering space by Hurwitz \cite{Hurwitz}. 
The number $n_i$ is called the {\em branching index\/} of $Q_i$. 

In the following, we will use the expression 
$(n, \theta_1/n + \cdots + \theta_b/n)$ 
in place of $[g,n;\theta_1, \ldots, \theta_b]$. 
This data $(n, \theta_1/n + \cdots + \theta_b/n)$ 
is called the {\em total valency\/}, which is introduced by 
Ashikaga and Ishizaka \cite{Ashikaga-Ishizaka}. 
In the above data, we call $\theta_i/n$ the {\em valency\/} 
of $Q_i$, and often rewrite this by an irreducible fraction. 
We remark that the denominator of the reduced $\theta_i/n$ 
is equal to the branching index $n_i$ of $Q_i$, and the  numerator of the reduced 
$\theta_i/n$ is well-defined modulo $n_i$. 
If $k$  is an integer prime to $n$ and $f = (n, m_1/n_1 + \cdots + m_b/n_b)$, then 
$f^k = (n, (k^* \cdot m_1)/n_1  + \cdots + (k^* \cdot m_b)/n_b)$ where 
$k^*$ is an integer such that $k \cdot k^* \equiv 1 $ mod $n$, 
and $k^* \cdot m_i$ is the remainder of $k^* m_i$ modulo $n_i$.  
\section{A discussion on branching indices}
Let $f$ be an order $n$ periodic map of $\Sigma_g$ whose orbit 
space $\Sigma_g / f = \mathbb{S}^2(n_1, n_2, n_3)$. 
We assume that $n_1 \leq n_2 \leq n_3$. 
By the Riemann-Hurwitz formula, we see 
\begin{equation}\label{eqn:R-H}
2(g-1) = n \left( 1 - \left( \frac{1}{n_1} + \frac{1}{n_2} + \frac{1}{n_3} 
\right) \right). 
\end{equation}
The branching indices $n_1, n_2, n_3$ satisfy the Harvey's $\lcm$ condition
\cite{Harvey}, 
that is, 
$$\lcm\{ n_1, n_2 \} = \lcm\{ n_2, n_3 \} = \lcm\{ n_3, n_1 \} = n.$$ 

\begin{lemma}\label{lem:k2k3}
%
Let $k_2 = n/n_2$, $k_3 = n/n_3$, then we see: \\
{\rm (i)} $n = \frac{n_1}{n_1 - 1} (2g+(k_2+k_3-2))$, \\
{\rm (ii)} $k_2 \geq k_3$, \\
{\rm (iii)} $k_2, k_3$ are divisors of $n_1$, \\
{\rm (iv)} $\gcd \{ k_2, k_3 \} = 1$, \\
{\rm (v)} $k_2+k_3 \leq n_1+1$. 
\end{lemma}
\begin{proof}
(i) is valid by (\ref{eqn:R-H}). \\
(ii) is valid by $n_2 \leq n_3$. \\
(iii) Let $k_1 = n/n_1$. 
Since $k_1 n_1 = k_2 n_2 = n = \lcm \{n_1, n_2 \}$, $k_1$ and $k_2$ are prime 
each other. 
By the equation $k_1 n_1 = k_2 n_2$, 
we see that $k_2$ is a divisor of $n_1$. 
By the same way, we see that $k_3$ is a divisor of $n_1$. \\
(iv) Since $k_2 n_2 = k_3 n_3 = n = \lcm \{n_2, n_3 \}$, 
$k_2$ and $k_3$ are prime each other. \\
(v) If $n_1=2$, then we see $k_2,k_3 \leq 2$ by (iii). 
Since $k_3 = 1$ by (iv), we see $k_2 + k_3 \leq n_1 + 1$. 
We assume $n_1 \geq 3$. 
If $k_2=n_1$ then $k_3=1$ by (iii) (iv). 
Therefore $k_2 + k_3 = n_1 + 1$. 
If $k_2 \not= n_1$ then $k_2 \leq n_1/2$. 
Moreover $k_3 \leq k_2$ by (ii), hence, 
we see $k_2 + k_3 \leq 2 k_2 \leq n_1 < n_1 + 1$. 
\end{proof}

\begin{theorem}\label{thm:n1-order-1}
%
The inequality 
$\displaystyle \frac{2n_1}{n_1-1} g \leq n \leq 
\frac{2n_1}{n_1-1}g + n_1$ 
is valid. 
\end{theorem}
\begin{proof}
Since $n_2, n_3 \leq n$, 
$$
2(g-1) = n \left( 1 - \left( \frac{1}{n_1} + \frac{1}{n_2} + \frac{1}{n_3} 
\right) \right) \leq 
n \left(1 - \left( \frac{1}{n_1} + \frac{2}{n} \right) \right) 
= n \left(\frac{n_1 - 1}{n_1} \right) -2 
$$
by the equation (\ref{eqn:R-H}), 
hence, 
$$
n \geq \frac{2n_1}{n_1 - 1} g.
$$
On the other hand, by (i) (v) of the previous Lemma, we see 
$$n = \frac{n_1}{n_1 - 1} (2g+(k_2+k_3-2)) \leq \frac{2n_1}{n_1-1} g + n_1.$$
\end{proof}

\begin{theorem}\label{thm:n1-order-2}
%
For an integer $N \geq 3$, we assume $g > (N-1)N(N+1)/2$. 
Then 
$$
n_1 = N
\Longleftrightarrow
\frac{2N}{N-1} g \leq n < \frac{2(N-1)}{N-2} g. 
$$
\end{theorem}
\begin{remark}
1. Because $(N-1)N(N+1)/2$ is increasing for $N \geq 3$, 
we see that, for any $N'$ such that $3 \leq N' \leq N$, 
$$
n_1 = N'
\Longleftrightarrow
\frac{2N'}{N'-1} g \leq n < \frac{2(N'-1)}{N'-2} g 
$$
under the assumption of the above Theorem. \\
2. In the case where $g = \frac{(N-1)N(N+1)}{2}$, 
there is a periodic map of order $\frac{2N}{N-1}g = N^2(N+1)$ 
such that $n_1 = N+1$ and whose valency data is 
$$
\left( N^2(N+1), \frac{N}{N+1} + \frac{N-1}{N^2} + \frac{1}{N^2(N+1)} 
\right).
$$
\end{remark}
\begin{proof}
We assume $n_1 = N$. 
By Theorem \ref{thm:n1-order-1}, 
$\frac{2N}{N-1} g \leq n \leq \frac{2N}{N-1}g + N$. 
By the assumption 
$g > (N-1)N(N+1)/2>(N-2)(N-1)N/2$, 
we see $\frac{2N}{N-1}g + N < \frac{2(N-1)}{N-2}g$. 
Therefore $\frac{2N}{N-1}g \leq n < \frac{2(N-1)}{N-2}g$. 

On the reverse order, we assume 
$\frac{2N}{N-1} g \leq n < \frac{2(N-1)}{N-2} g$. 
By Theorem \ref{thm:n1-order-1}, we see $n \geq \frac{2n_1}{n_1 -1} g$. 
If $n_1 \leq N-1$ then 
$\frac{2 n_1}{n_1 - 1} \geq \frac{2(N-1)}{N-2}$, 
hence $n \geq \frac{2(N-1)}{N-2}g$, 
which contradicts $\frac{2N}{N-1} g \leq n < \frac{2(N-1)}{N-2} g$. 
Therefore $n_1 \geq N$. 

Here, we show the following Lemma. 
\begin{lemma}
When $N \geq 2$, if $n \geq \frac{2N}{N-1}g$ then $n_1 \leq 3N-1$. 
\end{lemma}
\begin{proof}
By the equation (\ref{eqn:R-H})
$$
\frac{2g-2}{n} = 1 - \left( \frac{1}{n_1} + \frac{1}{n_2} + \frac{1}{n_3} \right).  
$$
By the assumption that $n_1 \leq n_2 \leq n_3$ and $n \geq \frac{2N}{N-1}g$, 
we see 
$$
\frac{(2g-2)(N-1)}{2Ng} \geq 1 - \frac{3}{n_1}. 
$$
By this inequality, we get an evaluation 
$n_1 \leq \frac{3g}{g+N-1} N < 3N$. 
Therefore, we conclude $n_1 \leq 3N-1$. 
\end{proof}
Here we assume that $n_1 \geq N + 1$. 
By Theorem \ref{thm:n1-order-1} and the assumption $n \geq \frac{2N}{N-1}g$, 
we see 
$$
\frac{2N}{N-1}g \leq \frac{2n_1}{n_1-1} g + n_1.
$$
By the above inequality and 
the inequality $\frac{2N}{N-1} > \frac{2n_1}{n_1-1}$ 
obtained from the assumption $n_1 \geq N + 1$, 
we see 
\begin{equation} \label{eqn:g}
g \leq \frac{(N-1)n_1(n_1-1)}{2(n_1-N)} . 
\end{equation}
By the above Lemma, $n_1 \leq 3N -1$. 
When $N \geq 3$, by the inequality $3N-1 \leq N^2$ and 
the assumption $n_1 \geq N+1$, we see $N+1 \leq n_1 \leq N^2$. 
From this inequality, $(n_1 - (N+1))(n_1 - N^2) \leq 0 $, 
that is, $n_1^2 - (N^2 + N +1)n_1 + N^2(N+1) \leq 0$, 
hence $n_1^2 - n_1 \leq N(N+1)(n_1 - N)$. 
By dividing the last inequality by $n_1 - N > 0$, we obtain 
$$
\frac{n_1(n_1-1)}{n_1-N} \leq N(N+1). 
$$
By this inequality and the inequality (\ref{eqn:g}), we see 
$$
g \leq \frac{(N-1)N(N+1)}{2}, 
$$
which contradicts the assumption. 
Therefore $n_1 \leq N$. 

We get a conclusion $n_1 = N$. 
\end{proof}

\section{The uniqueness by the order}

\begin{maintheorem}
Let $g > 30$ and $n > 8g/3$. 
If there is a periodic map whose order is $n$, then 
this periodic map is unique up to conjugacy and power. 
\end{maintheorem}
\begin{remark}
In the sentence, 
{\em if the genus $g$ is sufficiently large and $n > Mg$ (or $n \geq Mg$) 
then the order determines the conjugacy class of the periodic map 
up to power,\/} 
the condition of the order $n > 8g/3$ is best possible. 
In fact, when the genus $g = 3(2k+1)$, 
there are two periodic map of order $n = 8g/3$ whose total valencies are 
$$
\frac{1}{4} + \frac{1}{16k+8} + \frac{12k+5}{16k+8}, \qquad
\frac{3}{4} + \frac{1}{16k+8} + \frac{4k+1}{16k+8} 
$$
and any power of the first one is not conjugate to the second one. 
\end{remark}

\begin{proof}
By \cite{Kasahara}, the periodic map $f$ satisfying the condition 
in this Theorem is irreducible, that is 
$\Sigma_g / f = \mathbb{S}^2(n_1, n_2, n_3)$. 
As in \S3, we assume that $n_1 \leq n_2 \leq n_3$. 
Since we already discussed the case where $n \geq 3g$ in 
\cite{Kulkarni} and \cite{Hirose}, we assume that $n < 3g$. 
Since $g > 30 = \frac{4 \cdot (4^2-1)}{2}$, 
$\frac{2 \cdot (4-1)}{4-2} g = 3g > n 
> \frac{8}{3} g = \frac{2 \cdot 4}{4-1} g$, 
$n_1=4$ by Theorem \ref{thm:n1-order-2}. 
Let $k_2 = n/n_2$ and $k_3=n/n_3$. 
By Lemma \ref{lem:k2k3}, there are $3$ cases 
$(k_2, k_3)$ $=$ $(4, 1)$, $(2, 1)$, $(1, 1)$. 

(1) $(k_2, k_3) = (4, 1)$\ :\ 
By (i) of Lemma \ref{lem:k2k3}, 
the order $n = \frac{8}{3}g + 4 > \frac{8}{3} g$. 
Since $n$ is an integer, $g$ should be a multiple of $3$. 
Let $g = 3l$, then $n=8l+4$, 
$n_2 = n/k_2 = 2l+1$, and $n_3 = n/k_3 = 8l+4$. 
We determine the numerators $a, b, c$ of the valency data 
$$
\frac{a}{4} + \frac{b}{2l+1} + \frac{c}{8l+4}. 
$$
Since the branch point corresponding to $c/8l+4$ is 
the image of a fixed point of $f$ by $\pi_f$, 
we fix $c=1$ by taking a proper power of the periodic map $f$. 
Since $a/4$ is an irreducible fraction, $a = 1$ or $3$. 
If $a=3$, then $2b=l$. If $a=1$, then $2b = 3l+1$. 

When $l$ is even, we put $l=2m$. 
If $a=3$, then $b=m$. If $a=1$, then $2b=6m+1$. 
Therefore, $a=3$ and the total valency should be 
$$
\frac{3}{4} + \frac{m}{4m+1} + \frac{1}{16m+4}. 
$$

When $l$ is odd, we put $l=2m+1$. 
If $a=3$, then $2b=2m+1$. If $a=1$, then $b=3m+2$. 
Therefore, $a=1$ and the total valency should be 
$$
\frac{1}{4} + \frac{3m+2}{4m+3} + \frac{1}{16m+12}. 
$$

(2) $(k_2, k_3) = (2, 1)$ \ :\ 
By (i) of Lemma \ref{lem:k2k3}, 
the order $n = \frac{4(2g+1)}{3} > \frac{8}{3}g$. 
Since $n$ is an integer, $2g+1$ should be a multiple of $3$. 
Therefore $g \equiv 1 \mod 3$. 
Let $g = 3l+1$, then 
$n = 8l+4$, $n_2 = 4l + 2$, and $n_3 = 8l+4$. 
We determine the numerators $a, b, c$ of the valency data 
$$
\frac{a}{4} + \frac{b}{4l+2} + \frac{c}{8l+4}. 
$$
Since the branch point corresponding to $c/8l+4$ is 
the image of a fixed point of $f$ by $\pi_f$, 
we fix $c=1$ by taking a proper power of the periodic map $f$. 
Since $a/4$ is an irreducible fraction, $a = 1$ or $3$. 
If $a=1$, then $b=3l+1$. If $a=3$, then $b=l$. 
Since $\frac{b}{4l+2}$ is also an irreducible fraction, 
$b$ should be an odd integer. 
If $l$ is an even integer, $b=3l+1$ and $a=1$. 
Hence, the total valency should be 
$$
\frac{1}{4} + \frac{3l+1}{4l+2} + \frac{1}{8l+4}. 
$$
If $l$ is an odd integer, $b=l$ and $a=3$. 
Hence, the total valency should be 
$$
\frac{3}{4} + \frac{l}{4l+2} + \frac{1}{8l+4}. 
$$

(3) $(k_2, k_3) = (1, 1)$ \ :\ 
By (i) of Lemma \ref{lem:k2k3}, 
the order $n = \frac{8}{3} g$, which contradicts the condition 
$n > \frac{8}{3}g$. 
\end{proof}

By the proof of the above Theorem, \cite{Kulkarni} and \cite{Hirose} 
we see: 
\begin{cor}
Let $g > 30$ and $n > 8g/3$. 
If there is a periodic map $f$ of $\Sigma_g$ whose order is $n$, then 
$f$ is conjugate to a power of one of periodic maps listed on 
Table {\rm \ref{tab:big}}. 
\end{cor}
\begin{table}[ht]
\caption{}
\label{tab:big}
\begin{center}
\begin{tabular}{c|c}
\noalign{\hrule height0.8pt}
genus $g$ & total valency \\
\hline
{\rm arbitrary} & 
$\left(4g+2, \; \displaystyle\frac{1}{2}+\frac{g}{2g+1}+\frac{1}{4g+2} \right)$ \\
\hline
{\rm arbitrary} & 
$\left(4g, \; \displaystyle\frac{1}{2}+\frac{2g-1}{4g}+\frac{1}{4g} \right)$ \\
\hline
$3k$ & 
$\left(3g+3, \; \displaystyle\frac{2}{3}+\frac{k}{g+1}+\frac{1}{3g+3} \right)$ \\
\hline
$3k+1$ & 
$\left(3g+3, \; 
\displaystyle\frac{1}{3}+\frac{2k+1}{g+1}+\frac{1}{3g+3} \right)$ \\
\hline
$3k$ {\rm or} $3k+1$ & 
$\left(3g, \; \displaystyle\frac{1}{3}+\frac{2g-1}{3g}+\frac{1}{3g} \right)$ \\
\hline
$3k+2$ &  
$\left(3g, \; \displaystyle\frac{2}{3}+\frac{g-1}{3g}+\frac{1}{3g} \right)$ \\
\hline
$6m$ &
$\left(\displaystyle\frac{8}{3}g+4, \; \displaystyle\frac{3}{4}+
\frac{m}{4m+1}+\frac{1}{16m+4} \right)$ \\
\hline
$6m+3$ & 
$\left(\displaystyle\frac{8}{3}g+4, \; \displaystyle\frac{1}{4}+
\frac{3m+2}{4m+3}+\frac{1}{16m+12} \right)$ \\
\hline
$6m+1$ &
$\left(\displaystyle\frac{4(2g+1)}{3}, \; \displaystyle\frac{1}{4}+
\frac{6m+1}{8m+2}+\frac{1}{16m+4} \right)$ \\
\hline
$6m+4$ &
$\left(\displaystyle\frac{4(2g+1)}{3}, \; \displaystyle\frac{3}{4}+
\frac{2m+1}{8m+6}+\frac{1}{16m+12} \right)$ \\
\noalign{\hrule height0.8pt}
\end{tabular}
\end{center}
\end{table}
\section{A proof of Wiman's Theorem \cite{Wiman}}
We provide a short proof of Wiman's Theorem \cite{Wiman} using the argument in \S 3.

\begin{theorem}
When $g \geq 2$, the order of any periodic map of $\Sigma_g$ is at most $4g+2$.
\end{theorem}
In the following, it is the subcase III)-iii) which
is simplified by the argument mentioned above and 
seems to have been most involved.
While the treatment of the other subcases is standard, we include them for completeness.
\par
\begin{proof}
Let $n$ be the order of a periodic map $f$ of $\Sigma_g$, 
and $\Sigma_g / f = \Sigma_{g'}(n_1, \ldots, n_j)$, 
where $n_1 \leq n_2 \leq \cdots \leq n_j$. 
By the Riemann-Hurwitz formula, 
\begin{equation}\label{eq:RH}
\frac{2(g-1)}{n} 
= 2 (g' -1) + j - \left( \frac{1}{n_1}+\cdots+\frac{1}{n_j} \right). 
\end{equation}
I) When $g' \geq 2$, the RHS of (\ref{eq:RH}) $\geq 2 (g' - 1) \geq 2$.  
Therefore $2(g-1)/n \geq 2$, that is, $g-1 \geq n $, 
then we see $n \leq 4g+2$. 
\newline
II) We discuss the case where $g' = 1$. 
If $j=0$, then $g=1$, which contradicts the assumption 
$g \geq 2$. Hence, $j \geq 1$. 
Since each $n_i \geq 2$, we see $1/n_1 + \cdots + 1/n_j \leq j/2$.
Therefore, the RHS of (\ref{eq:RH}) $\geq j - j/2 = j/2$, 
hence $2(g-1)/n \geq j/2$, and we see $n \leq 4(g-1)/j$. 
This shows that $n \leq 4g+2$ in this case. 
\newline
III) We discuss the case where $g'=0$. 
We first note that Proposition \ref{prop:cond-data} implies 
$j \geq 3$. 
\newline
i) $j \geq 5$\ : in this case, 
the RHS of (\ref{eq:RH}) $\geq -2 + j/2 \geq 1/2$, 
hence  $n \leq 4(g-1)$. Therefore, $n \leq 4g+2$. \newline
ii) $j = 4$\ : 
We multiply $n$ on both sides of (\ref{eq:RH}) and change $n/n_i$ into $k_i$ 
for $i \not=1$, then we see, 
$$
\begin{aligned}
n &= \frac{n_1}{2n_1 - 1}(2g-2+k_2+k_3+k_4) \\
&= \frac{n_1}{2n_1 - 1}(2g-2) + \frac{1}{2n_1-1}(n_1 k_2+n_1 k_3 + n_1 k_4) \\
&\leq \frac{n_1}{2n_1 - 1}(2g-2) + \frac{3n}{2n_1-1}. 
\end{aligned}
$$
In the last inequality, we use $n_1 k_i \leq n_i k_i = n$. 
Since 
$$
n - \frac{3n}{2n_1 - 1} = n \left( 1 - \frac{3}{2n_1 - 1} \right) = n \frac{2(n_1-2)}{2n_1-1}, 
$$
in the case where $n_1 \geq 3$, we obtain 
$$
n \leq \frac{n_1}{n_1-2}(g-1). 
$$
Because $n_1/(n_1-2) \leq 3$, we have $n \leq 3(g-1) < 4g+2$. 
In the case where $n_1=2$, we assume that 
$n \geq 4g+2 \geq 4 \cdot 2 + 2 = 10$. 
By Harvey's lcm condition \cite{Harvey}, we remark 
\begin{equation}\label{lcm:Harvey}
n = \lcm \{n_1, n_2, n_3 \} 
= \lcm \{ n_1, n_2, n_4 \} 
= \lcm \{ n_1, n_3, n_4 \} 
= \lcm \{ n_2, n_3, n_4 \}. 
\end{equation}
If all $n_i$ are $2$ or $3$, then $n \leq 6$ by (\ref{lcm:Harvey}) 
which contradicts the assumption $n \geq 10$. 
Therefore, there are some $n_i$'s such that $n_i \geq 4$. 
By (\ref{lcm:Harvey}), there are at least $2$ $n_i$'s 
such that $n_i \geq 4$, hence $n_3, n_4 \geq 4$. 
We see 
$n_1 k_3 = \frac{n_1}{n_3} n_3 k_3 = \frac{n_1}{n_3} n \leq \frac{n}{2}$, 
and in the same way, we see $n_1 k_4 \leq \frac{n}{2}$. Therefore, 
$$
\begin{aligned}
n &= \frac{n_1}{2n_1 - 1}(2g-2) + \frac{1}{2n_1-1}(n_1 k_2+n_1 k_3 + n_1 k_4) \\
&\leq \frac{n_1}{2n_1 - 1}(2g-2) + \frac{1}{2n_1-1} (n + \frac{n}{2} + \frac{n}{2}) 
= \frac{4g-4}{3} + \frac{2}{3} n.  
\end{aligned}
$$
Hence we have $n \leq 4g-4$ contradicting the assumption $n \geq 4g+2$. 
We conclude $n < 4g+2$. 
\newline
iii) $j=3$\ : \ At first, we show the following Lemma: 
\begin{lemma} \label{lem:n1-eval}
If the orbit space of a periodic map $f$ of $\Sigma_g$ 
is a $2$-sphere with $3$ branch points, and let $n_1$ be 
the minimal branching index, then $n_1 \leq 2g+1$. 
\end{lemma}
\begin{proof}
By (\ref{eq:RH}), we see 
$$
2(g-1) = n \left( 1 - \left( \frac{1}{n_1}+\frac{1}{n_2}+\frac{1}{n_3} \right) \right) = 
n - (k_1 + k_2 + k_3). 
$$
From the above equation, we have 
\begin{equation}\label{eq:lower-bound}
n= 2g + (k_1+k_2+k_3) - 2 \geq 2g+1. 
\end{equation}
On the other hand, by the assumption $n_1 \leq n_2 \leq n_3$, 
we see $1/n_3 \leq 1/n_2 \leq 1/n_1$. 
It follows that $1-(1/n_1 + 1/n_2 + 1/n_3) \geq 1 - 3/n_1$, 
and by the equation (\ref{eq:RH}), 
$$
2(g-1) = n \left( 1- \left( \frac{1}{n_1}+\frac{1}{n_2}+\frac{1}{n_3} \right) 
\right) 
\geq n \left( 1-\frac{3}{n_1} \right). 
$$
If $\left( 1 - \dfrac{3}{n_1} \right) \leq 0$, then $n_1 \leq 3$. 
Since $g \geq 2$, we see $2g+1 \geq n_1$. 
If $\left( 1 - \dfrac{3}{n_1} \right) > 0$, 
by the equation (\ref{eq:lower-bound}), we have 
$$
2(g-1)\geq n \left( 1-\frac{3}{n_1} \right) \geq (2g+1) 
\left( 1-\frac{3}{n_1} \right).
$$
Therefore $2g+1 \geq n_1$. 
\end{proof}

By Theorem \ref{thm:n1-order-1}, we have 
$$
n \leq \frac{2 n_1}{n_1 - 1}g + n_1. 
$$
By this inequality, we see 
$$
\begin{aligned}
(4g+2)-n &\geq (4g+2)-\left( \frac{2n_1}{n_1 - 1}g+n_1 \right) = 
\left( 4 - \frac{2 n_1}{n_1-1} \right) g + (2-n_1) \\
& = \frac{2n_1-4}{n_1-1} g + (2-n_1) = (n_1 - 2) \left( \frac{2}{n_1-1}g -1 
\right), 
\end{aligned}
$$
Since the branching index is at least $2$, $n_1 - 2 \geq 0$. 
By Lemma \ref{lem:n1-eval} we have $n_1 \leq 2g+1$, 
that is, $\frac{2}{n_1-1}g-1 \geq 0$. 
By the above inequality, we conclude that $4g+2 \geq n$. 

\end{proof}
\subsection*{Acknowledgments}
The authors would like to express their gratitude to Professor Gromadzki
for informing about their preprint \cite{BCGH}.

\baselineskip=13pt
\noindent
{\large Susumu Hirose} \\
Department of Mathematics, \\ 
Faculty of Science and Technology, \\ 
Tokyo University of Science, \\
Noda, Chiba, 278-8510, Japan \\
E-mail: hirose\b{ }susumu@ma.noda.tus.ac.jp \\
\\
{\large Yasushi Kasahara} \\
Department of Mathematics, \\ 
Kochi University of Technology, \\
Tosayamada, Kami City, Kochi \\
782-8502 Japan \\
E-mail: kasahara.yasushi@kochi-tech.ac.jp
\end{document}